\documentclass[11pt]{amsart}
\usepackage{amsmath,amssymb,enumerate,amsfonts,amsthm,graphicx,color}
\usepackage{eucal,bbm,amscd}
\usepackage{mathrsfs}
\usepackage[all]{xy}
\usepackage{tikz}
\usetikzlibrary{arrows,snakes,backgrounds}
\usetikzlibrary{topaths}
\usepackage{tikz}  
\usetikzlibrary{arrows,snakes,backgrounds}
\usetikzlibrary{topaths} 
\usetikzlibrary[topaths] 

\newtheorem{thm}{Theorem}[section]
\newtheorem{cor}[thm]{Corollary}

\newtheorem{con}[thm]{Conjecture}

\newtheorem{prop}[thm]{Proposition}
 
\theoremstyle{definition}
\newtheorem{defn}[thm]{Definition}
\theoremstyle{definition}
\newtheorem{rem}[thm]{Remark}
\theoremstyle{definition}
\newtheorem{exm}[thm]{Example}
\theoremstyle{definition}
 
  \newtheorem{thevarthm}[thm]{\varthmname}

\newenvironment{varthm*}[1]{\trivlist\item[]{\bf #1.}\it}{\endtrivlist}

\usepackage[naturalnames=true]{hyperref}

 \title[The weak Lefschetz property of a special class of Artinian $\dots$]{The weak Lefschetz property of a special class of Artinian algebras over fields of positive characteristic}
\author{Hassan Haghighi and Sepideh Tashvighi}
\address{Hassan Haghighi, Sepideh Tashvighi, \ Faculty of Mathematics, K. N. Toosi
University of Technology, Tehran, Iran.} \email{haghighi@kntu.ac.ir,
sepid.tashvighi@email.kntu.ac.ir}

\begin{document}

\maketitle

\begin{abstract}

 In this paper, we study the dependence of the weak Lefschetz property
 of
  algebras defined by a special class of monomials ideals in a polynomial ring with coefficient in a field,
    to the characteristic of the base field.
\end{abstract}

 \section{Introduction}
Let $\mathbb{K}$ be an arbitrary infinite field, and let
$R=\mathbb{K}[x_1, \dots, x_r]$ be the polynomial ring  with
standard grading. Let $I$ be a homogeneous ideal in $R$ such that
$A=R/I$ is an Artinian algebra. This is equivalent to say that the
radical of $I$ is equal to $(x_1, \dots, x_r)$, or,  $A$ can be
written as $\bigoplus_{i=0}^e A_i$.

A significant property which a standard graded Artinian
$\mathbb{K}$-algebra may pose is the weak Lefschetz property (WLP
for short). A standard graded $\mathbb{K}$-algebra has the WLP, if
there exists a linear form in $R$ such that for each $0\le i \le
e-1$, the multiplication map $\times \ell:A_i \longrightarrow
A_{i+1}$ has maximal rank, i.e., it is injective or surjective. This
property, not only depends to algebraic structure of the algebra
$A$, but also depends on the characteristic of the base field
$\mathbb{K}$.

In addition to  intrinsic significance of the WLP for a standard
graded Artinian $\mathbb{K}$-algebra, this property is closely
related to some problems in other disciplines of mathematics. For
example the presence of this property is related to existence of
finite projective planes \cite{CMNZ, Maeno}, has connection with
some special family of curves in Algebraic Geometry \cite{Miro}, it
is associated with the problem of enumerating  the plane partitions
in combinatorics \cite{DU, LZ}.

 Even though, this property has
a simple definition, but establishing it for a general standard
graded Artinian $\mathbb{K}$-algebra, is not an easy task. As a
consequences of this fact, classifying all standard graded
$\mathbb{K}$-algebra which pose this property would be a hard
problem. This forces to look for this property in special classes of
standard graded Artinian $\mathbb{K}$-algebras. Among  such
algebras, those for which $I$ is a  monomial ideal, are the most
accessible and excellent  ones.

When the characteristic of the base field $\mathbb{K}$ is zero,
Stanley \cite{Stan1} showed that the algebra $\mathbb{K}[x_1, \dots,
x_r]/(x_1^{d_1}, \dots, x_r^{d_r})$, where $d_is$ are greater than
1, enjoys the SLP, and hence the WLP, but when the characteristic of
the base field is positive, this result is no longer hold. It can be
easily shown that whenever $\mathrm{char\ } \mathbb{K }=p$, and
$r\ge 3$, the algebra $\mathbb{K}[x_1,\dots, x_r]/(x_1^p, \dots,
x_r^p)$, does not have the WLP, while it poses this property for
$r\le 2$. In \cite{MI_MI}, it is shown that if the ideal
$I=(x_1^{d_1}, \dots, x_r^{d_r})$, where $d_1\ge \dots \ge d_r\ge
2$, satisfies the condition $d_1
>[t/2]$, where $t=d_1+d_2+\dots+d_r-n$, then $R/I$ has the WLP,
regardless of the characteristic of $\mathbb{K}$. Moreover, in
\cite[Proposition 3.5]{D}, it is shown that if the $\mathrm{char\
}\mathbb{K}=p>0$ and $d_1\le \lceil t/2 \rceil$, and if $d_2 \le p
\le d_1$ or for some positive integer $m$, the condition $d_1 \le
p^m\le \lceil t/2 \rceil$ holds, then $R/I$ fails to have the WLP.
In \cite{BK}, by a geometric method, the WLP of the algebra
$\mathbb{K}[x,y,z]/(x^d,y^d,z^d)$, in terms of integer $d$ and the
characteristics of the field $\mathbb{K}$ is determined. In
\cite{LZ}, it is proved that  if  $p=\mathrm{char\ }\mathbb{K}$ is a
prime divisor of the number of plane partitions $M(a,b,c)$,  then
the $\mathbb{K}$-algebra
$A=\mathbb{K}[x,y,z]/(x^{a+b},y^{a+c},z^{b+c})$
does not have the WLP. In \cite[Theorem 3.8]{LN}, a complete
classification of the SLP for all monomial ideals $I$, for which $I$
is a complete intersection and
 $\mathbb{K}>0$, is given. In \cite{DU}, by tools
which have been developed in \cite{JJ},  the WLP of the monomial
ideals of the form
$I=(x^{t+\alpha}y^{t+\beta}z^{t+\gamma},x^{\alpha}y^{\beta}z^{\gamma})$,
known as almost complete intersection, are investigated. In
particular, those characteristics that these type of ideals may fail
to have the WLP are determined in terms of the exponents $t,\alpha,
\beta, \gamma$.

In \cite{JM}, the weak Lefschetz property of algebras defined by
another class of monomial ideals which are in the form
\begin{align*}
I_{r,k,d}=(x_1^k, \dots, x_r^k) +(\mathrm{all \  squarefree
monomilas \ of \ degree\ } d)
\end{align*}
are studied and the WLP of these type of ideals, whenever $d=2$ and
$d=3$, are established  (\cite[Theorem 3.3]{JM}), and for general
case is stated as the following conjecture (see \cite[Conjecture
3.4.]{JM})
\begin{con}\label{conj}
Consider the  algebra $R/I_{r,k,d}$, where the ideal $I_{r,k,d}$, is
defined as
\begin{center}
$\left(  x_1^{k}, \dots, x_r^{k}\right) +\left( \mathrm{ all \
square \ free \ monomials \ of \ degree} \ d\right)$.
\end{center}
Then
\begin{flushleft}
(a) If $d=4$, then it has the WLP if and only if $k \mod 4$  is 2 or 3.\\
(b) If $d=5$, then  the WLP fails.\\
(c) If  $d=6$, then the WLP fails.
\end{flushleft}
\end{con}\label{JU_MI_NA}
The proof of part $(b)$ of \cite [Theorem 3.3]{JM}, and the above
conjecture has motivated us to investigate the presence of the WLP
for a class of monomial  ideals which their generators are nearly
similar to the generators of $I_{r,k,d}$. I.e., we consider the
ideals of type
 \begin{align*}
 I=\left(  x_1^{\alpha_1}, \dots, x_r^{\alpha_r}\right) +\left(\textrm{ all\ square \ free\ monomials\ of \ degree}\ d
 \right),
\end{align*}
where $2 \le \alpha_i \le 4$ for $1 \le i \le r$, and study the WLP
behavior of $R/I$ with respect to small prime numbers.

 Our main results
are:

\noindent{\bf Theorem A.} Let $I$ be as the above ideal, where
 $ 2 \le \alpha_i \le 3$ for all $d\geq 4$ and $1 \le i \le r$. Then
 $A=R/I$ does not have the WLP, whenever $\mathrm{char}\ \mathbb{K}=3$.

In particular, if all $\alpha_i$s, are equal to 2,  or are equal to
3, then Theorem A, implies the part $(a)$ of  the conjecture
\ref{JU_MI_NA} can not be  true, while it confirms what is claimed
in  parts $(b),(c)$ of \ref{conj} for a specific value.

 \noindent{\bf Theorem B.} Let $I$ be as the ideal, where
 $2 \le \alpha_i \le 4$ for all $d \geq 5$ and $1 \le i \le r$. Then
 $A=R/I$ does not have the WLP,  whenever $\mathrm{char}\ \mathbb{K}=2$.

Another result of this paper is a little bit different from the
other results. In fact, we determined  all characteristics of the
base field which a special type of ideals  define an algebra without
the WLP.

\begin{prop} 
Let
\begin{center}
$I^{'}=\left( x_1^2, \dots, x_r^2\right) +\left( \mathrm{all \
squarefree \ monomials \ of \ degree} \ d\right)$
\end{center}
be an ideal in $R$. Then $R/I^{'}$ is a level algebra and
$e=\mathrm{ Socle Degree }(I^{'})=d-1$. Moreover, $R/I^{'}$ doesn't
have the WLP in $\mathrm{char}\ \mathbb{K}= \mathit{p}$ whenever
$\mathit{p}$ is a prime number less than  $i+2$, where $1 \leq i
\leq \lceil r/2 \rceil$.
\end{prop}

The method of proof of  Theorem A and Theorem B, can be applied to
prove  the failure of the WLP for the following class of monomial
ideals:
\begin{align*}
J=\left( x_1^{\alpha}, \dots, x_r^{\alpha},
x_1^{\alpha-2}x_2^{2},x_1^2x_2^{\alpha -2},
\dots,x_{r-1}^{\alpha-2}{x_{r}}^{2},x_{r-1}^{2}{x_{r}}^{\alpha-2}
\right),
\end{align*}
provided $r\ge 4, \alpha \ge 5.$

 \noindent{\bf Theorem C.} Let $J$ be as the above ideal  in $R$.
Let $\mathrm{char}\ \mathbb{K}=2$. Then  $R/J$ does not have the
WLP.

\section{Preliminaries}

Let $A=R/I=\bigoplus_{i=0}^e A_i$ be a standard graded Artinian
$\mathbb{K}$ algebra. Then the function $h_i=\dim A_i$, for $0\le i
\le e$ is called the Hilbert function of $A$. Since $\dim A_i=0$ for
$i>e$, this function can be represented as an array
$\textbf{h}=(h_0,h_1, \dots, h_e)$, which is called the
$\textbf{h}$-vector of $A$. If we denote the maximal ideal of $A$ by
$\mathfrak{m}$, then the ideal
$$(0:_A \mathfrak{m})=\{ a\in \mid
a\mathfrak{m}=0\}=\mathcal{U}_0\oplus \mathcal{U}_1\oplus \dots
\oplus \mathcal{U}_e,$$ is called the socle of $A$. Since this ideal
gathers the annihilators of $\mathfrak{m}$, its structure is closely
related to the WLP of the algebra $A$. Since $A_{e+1}=0$, it is
clear that $A_e\subset (0:_A \mathfrak{m})$. The integer $e$ is
called the socle degree of $A$. Moreover, if $\mathcal{U}_i=0$ for
all $i <e$, then $A$ is called a level algebra.

\begin{defn}
Let $\ell$ be a general linear form in $R$. We say that the Artinian
ring $A$ has the {\it weak Lefschetz property} (WLP for short) if
the homomorphism induced by multiplication by $\ell$,
\begin{center}
$\times \ell: A_{i}\longrightarrow A_{i+1}$,
\end{center}
has the maximal rank for every $i$, $0\le i \le e-1$ (i.e., it is
injective or
surjective). In this case, the linear form $\ell$ is called the Lefschetz element for $A$.\\
We say that $A$ has the {\it strong Lefschetz property} the (SLP) if
\begin{center}
$\times \ell^{d}: A_{i}\longrightarrow A_{i+d}$
\end{center}
has the maximal rank for every $i$ and $d$, with $0\le i \le e-2$
and $1 \le d \le e-1$.
\end{defn}
If a standard graded $\mathbb{K}$-algebra $A$, has a Lefschetz
element $\ell$, then it can be shown that there is a Zariski open
set in $\mathbb{P}^{r-1}=\mathbb{P}(\mathbb{K}[x_1,\dots,x_r]_1)$
which parameterizes all Lefschetz elements of $A$.

In \cite[Proposition 2.2]{JJ}, it is shown if the field $\mathbb{K}$
is infinite, and the monomial Artinian $I$  satisfies the WLP, then
the linear form $\ell=x_1+\dots+x_r$ would be a Lefschetz element
too. This simplifies the checking for posing or failure of the WLP.
Moreover, in \cite[Proposition 4.3]{LN}, the assumption of being
infinite for $\mathbb{K}$ has been weaken. This will allows us to
use finite fields to construct examples or counterexamples for
posing of failure of the WLP.


\begin{prop}(\cite[Proposition 4.3]{LN})
Let $\mathbb{K}$ be a field and let $\mathbb{K}^{'}$ be an extension
field of $\mathbb{K}$. Let $I\subset \mathbb{K}[x_1, \dots, x_r]$ be
a monomial ideal. Then the following are equivalent.
\begin{itemize}
\item[(a)] $A=\mathbb{K}[x_1, \dots, x_r]/I$ has the WLP.
\item[(b)] $A^{'}=A\otimes_{\mathbb{K}} \mathbb{K}^{'}$ has the WLP.
\item[(c)] $x_1+\dots+x_r$ is a weak Lefschetz element of $A$.
\item[(d)] $x_1+\dots+x_r$ is a weak Lefschetz element of $A^{'}$
\end{itemize}

\end{prop}


\section{Main Results}

In this section, we prove the main results of this paper. The
Artinian algebras that we consider are defined by the ideals of the
following form.
\begin{align*}
I=\left(  x_1^{\alpha_1}, \dots, x_r^{\alpha_r}\right) +\left( all \
square \ free \ monomials \ of \ degree \ d\right),
\end{align*}
where $2\le \alpha_i \le 4$ and $2 \le d \le r$.

 The main idea of the proof of the main
results, is to construct a homogenous polynomial $f \in A=R/I$, such
that for suitable indices $i$, the map $\times \ell: A_i
\longrightarrow A_{i+1}$ fails to be injective.


\begin{thm}({\bf Theorem A.})\label{M}
Let $I$ be as above ideal, where  $d\geq 4$ and $ 2 \le \alpha_i \le
3$ for all $1 \le i \le r$. If  $\mathrm{char}\ \mathbb{K}=3$, then
 $A$ does not have the WLP.
\end{thm}
\begin{proof}
According to Remark \ref{r1}, $h_2\le h_3$.  
We show that the map $\times \ell: A_2 \longrightarrow A_3$ can not
be injective. To prove our claim, let
\begin{align*}
f =\sum_{ 1 \le i < m \le r } (-1)^jx_i^jx_m^{2-j}, \mathrm{ \ \
where \ } j=1,2.
\end{align*}
It is clear that $f$ is a nonzero element of $ A_2$. Moreover, in
$A_3$, the element $f \times \ell $, consists of the following
terms:
$$
\begin{array}{llllllllll}
& \mathrm{ for }& j=1,&  - x_i^2x_m,& -x_ix_m^2,& -x_ix_mx_{t},&
\mathrm { where }& t \neq i, m,\ 1 \le t \le r; & \\
& \mathrm{ for  }& j=2,& \ \ x_i^2x_m,&\ \ x_i^2x_k,&\ \
x_i^3,&\mathrm{ where }& 1 \le k < i < m \le r. & &
\end{array}
$$
Notice that for those $\alpha_i s$ which are equal to two, the above
terms can not be appeared in $f \times \ell$. Moreover, as can be
seen from the above expressions, for $j=1$ and $j=2$, the first two
terms  of the first row are respectively, additive inverses of the
first two terms of the second row. Hence, their sum in $f \times
\ell$ cancel each other.

Moreover,  each term $-x_ix_m^2$ in the first row, is the additive
inverse of a monomial $x_i^2x_k$, in the second row, and vice versa.
Hence, their sum in $f \times \ell$ is zero.

On the other hand, by our assumption, $x_i^3$ modulo $I$ is zero.
Hence, the only terms which will be remained in $f \times \ell$, are
in the form $-x_ix_mx_{t}$, where $t \neq i,m$ and $ 1 \le t \le r$.
We  count the occurrences of these monomials in $f\times \ell$ via
two methods.

 {\it First method.} We know that
the number of squarefree monomials of degree 3 in  $r$ variables  is
equal to ${r \choose 3}$. We label these  monomials as $a_{1},
a_{2}, \dots, a_{{r \choose 3}}$ and set $M= a_{1}+ a_{2}+ \dots+
a_{{r \choose 3}}$.

 {\it Second method.} By definition of $f$, the terms
of $f \times \ell$ are products of all $x_{i}x_{m}$ and the
different variables with $x_{i}$ and $x_{m}$ that vary  in $\{x_1,
\dots, x_r\}$. The possible number of  such monomials is equal to
\begin{align*}
(r-2) {r \choose 2}=\frac{r(r-1)(r-2)}{2}.
\end{align*}
We label these  monomials as $b_{1}, b_{2}, \dots, b_{(r-2) {r
\choose 2}}$.

 By comparing the results of these two methods, we
observe that the monomials $x_ix_mx_t$ of  $f \times \ell$, have a
coefficient equal to 3 in $A_{3}$, since
 $b_{1}+ b_{2}+ \dots+ b_{(r-2) {r \choose 2}}=3( a_{1}+ a_{2}+ \dots+ a_{{r \choose 3}})=3M$.
Therefore, $f \times \ell =0$ modulo $I$, while $f \neq 0$. This
means the kernel of the map $\times \ell: A_2\longrightarrow A_3$ is
nontrivial. Hence, $A$ does not have the WLP.
\end{proof}
 \begin{rem}\label{r1}
 Let $I$ be as the above ideal, where  $d\geq 4$ and $ 2 \le \alpha_i \le 3$ for all
$1 \le i \le r$. The number of monomials of degree $i$ in $R_i$, is
equal to $i+(r-1) \choose r-1$. Hence $h_2= {2+(r-1) \choose r-1}
\leq h_3= {3+(r-1) \choose r-1}$, but we should remove those
monomials which are multiples of the generators of $I$. Let $m$ be
the number of these  degree 2 terms  and $n=r-m$ be the number of
these degree 3 terms. Then we have
\begin{center}
$h_2= {2+(r-1) \choose r-1}-m$ and $h_3= {3+(r-1) \choose
r-1}-(m.r)-n$
\end{center}
If $h_2 \leq h_3$, we are nothing to do and it is the case that  we
need for our argument, but  if $h_3 < h_2$, then this would be the
exceptional cases which should be avoided. Hence, if we evaluate
both sides of this inequality, then
\begin{center}
$h_2= {2+(r-1) \choose r-1}-m>h_3= {3+(r-1) \choose
r-1}-(m.r)-n={3+(r-1) \choose r-1}-(m.r)-(r-m),$
\end{center}
which implies  $7r+6rm> r^3+12m$. Note that $m$ and $ n $ varies in
$\{0,1,\dots,r\}$.

Notice that whenever $r$ becomes larger, the above inequality is no
longer hold. Therefore, these exceptional cases happen whenever
 $m=r=4$.
\end{rem}
In \cite[Theorem 4.3]{JM}, it is proved that the Artinian algebra
defined by the ideal in the general form
\begin{align*}
I_{r,r}=(x_1^r, \dots, x_r^r, x_1x_2\cdots x_r),
\end{align*}
fails to have the WLP. As a special case of the above theorem, we
can state the following  result.
\begin{cor}
Let the ideal $I$ be as in Theorem \ref{M}. If $d=r$, then the
algebra $R/I$ does not have the WLP.
\end{cor}

Now we can state a similar result for the case $\mathrm{char}\
\mathbb{K}=2$.

\begin{thm}({\bf Theorem B.})\label{M2}
Let $I$ be  the ideal as in Theorem \ref{M}. Let $d \geq 5$ and $2
\le \alpha_i \le 4$ for all $1 \le i \le r$. Then
 $A$ does not have the WLP,  whenever $\mathrm{char}\ \mathbb{K}=2$.
\end{thm}
\begin{proof}
According to Remark \ref{r2}, $h_3\le h_4$. We show that the map
$\times \ell: A_3 \longrightarrow A_4$ can not be injective. To
prove this, let $f\in A_3$ be as follow
\begin{align*}
f=\sum_{1\le i<m\le r}x_i^jx_m^{3-j} \ \ \mathrm{ where \ } j=1,2,3.
\end{align*}
Then terms of the polynomial $f \times \ell$ can be grouped together
in terms of the value of $j$ and consists of
$$
\begin{array}{llllll}
\mathrm{for \ } j=1, & x_i^2x_m^2,& x_ix_m^3,&  x_ix_m^2x_t; &
 t \neq i, m, 1 \le  t  \le r, & \\
\mathrm{for \ } j=2, & x_i^3x_m,&  x_i^2x_m^2,&  x_i^2x_mx_t; &
 t \neq i, m, 1 \le  t  \le r, &
\\ \mathrm{for \ } j=3, &  x_i^4,&  x_i^3x_m,& x_i^3x_k;&
 1\le k < i<m   \le r.
\end{array}
$$
As the above list shows, the terms $x_i^2x_m^2$  appear  in
 $f \times \ell$ for $j=1$ and  $j=2$.
 Hence, the sum of these terms in  $f \times \ell$ has
coefficient 2, which by our assumption, this sum would become zero.
Moreover, the terms $x_i^3x_m$ exist in the rows $j=2$ and $j=3$.
Hence, the coefficient of their sum is 2. Therefore, these terms
would be killed in $f \times \ell$. As well as,  each monomial
$x_ix_m^3$, in the first row of the above list, is equal to a term
$x_i^3x_k$, in the third row and vice versa. Hence, their sum would
be zero in $f \times \ell$. On the other hand,
 $x_i^4$ is zero   modulo $I$. Hence, it
remains the terms $x_ix_m^2x_t$ and $x_i^2x_mx_t$ in $f \times
\ell$. We  count the number of occurrence of these terms in $f
\times \ell$ via two ways.

{\it First method.} The number of distinct squarefree monomials of
degree 3 in terms of $\{x_1, \dots, x_r\}$ is equal to ${r \choose
3}$, and since the exponent of one of the variables in each terms is
equal to 2, the total number of these type of terms is equal to
\begin{align*}
3 {r \choose 3}= \frac{r(r-1)(r-2)}{2}.
\end{align*}
We label these terms as $a_{1}, a_{2}, \dots, a_{3 {r \choose 3}}$,
and set $M=a_{1}+ a_{2}+ \dots+ a_{3 {r \choose 3}}$.

{\it Second method.} Each of one of the monomials  $x_ix_m^2x_t$ and
$x_i^2x_mx_t$ is obtained by multiplying $x_t \in \{x_1, \dots,
x_r\}$, where $t \neq i, m$, by $x_i x_m^{2}$ and $x_i^2 x_m$ of
$f$. Hence, the total number of these summands is equal to $2(r-2){r
\choose 2}=r(r-1)(r-2)$.
We label these monomials as $b_{1}, b_{2}, \dots, b_{2(r-2){r \choose 2}}$.\\
By comparing the results of these two methods of counting, one can
observe that all monomials in the $f \times \ell$, are multiplied by
two, since $b_{1}+ b_{2}+ \dots+ b_{2(r-2){r \choose 2}}=2(a_{1}+
a_{2}+ \dots+ a_{3 {r \choose 3}})=2M$. Hence, by our assumption,
the sum is equal to 0, and the proof completes.
\end{proof}
 \begin{rem}\label{r2}
 Let $I$ be as the above ideal, where  $d\geq 5$ and $ 2 \le \alpha_i \le 4$ for all
$1 \le i \le r$. Then an argument similar to the one used in the
Remark \ref{r1} shows that $h_3 \leq h_4$.
\end{rem}

Due to its importance, we state a special case of the above theorem
as a corollary.
\begin{cor}
Let the ideal $I$ be as in Theorem \ref{M2}. If $d=r$, then the
algebra $R/I$ does not have the WLP.
\end{cor}
Contrary to the previous results which in characteristics 2 and 3,
we showed that the algebras that we considered do not pose the WLP,
in the next result, we determine all primes that the Artinian
algebra defined a specific ideal may not pose the WLP.
\begin{prop} \label{1}
Let
\begin{center}
$I^{'}=\left( x_1^2, \dots, x_r^2\right) +\left( \mathrm{all \
squarefree \ monomials \ of \ degree} \ d\right)$
\end{center}
be an ideal in $R$. Then $R/I^{'}$ is a level algebra and
$e=\mathrm{ Socle Degree }(I^{'})=d-1$. Moreover, $R/I^{'}$ doesn't
have the WLP in $\mathrm{char}\ \mathbb{K}= \mathit{p}$ whenever
$\mathit{p}$ is a prime number less than  $i+2$, where $1 \leq i
\leq \lceil r/2 \rceil$.
\end{prop}
\begin{proof}
By our assumption on $I$, each homogeneous component $A_i$ of $A$ is
generated by squarefree monomials of degree $i$. Hence $h_i={r
\choose i}$. Therefore,
\begin{center}
$h_0 \leq h_1 \leq \dots \leq h_{\lceil r/2 \rceil}$.
\end{center}

We show that the map $\times \ell: A_i \longrightarrow A_{i+1}$ can
not be injective for any $1 \leq i \leq \lceil r/2 \rceil$. Let
\begin{center}
$f =\sum x_{j_1}x_{j_2} \cdots x_{j_i}$,
\end{center}
where $1\leq j_i \leq r-1$.\\
According to the definition of $f$, the nonzero terms of $f \times
\ell$ are in the form:
\begin{center}
$x_{j_1}x_{j_2} \cdots x_{j_i} x_m$
\end{center}
where $1\leq m \leq r$ and $m$ is not in $\{j_1,\dots, j_i\}$.

Now we want to count the occurrence of these terms in $A_{i+1}$.\\
{\it First method.} We know that the number of squarefree monomials
of degree $i+1$ in $R$ 
is equal to ${r \choose i+1}$. We label these terms as $a_{1},
a_{2}, \dots, a_{{r \choose i+1}}$,
and set $M= a_{1}+ a_{2}+ \dots+ a_{{r \choose i+1}}$.\\
{\it Second method.} By definition of $f$, these terms of $f \times
\ell$ are products of all $x_{j_1}x_{j_2} \cdots x_{j_i} $ and the
different variables with $x_m$ in $\{x_1,\dots, x_r\}\setminus
\{x_{j_1}, \dots, x_{j_i} \}$. The number of possible such terms is
equal to
\begin{center}
$(r-i) {r \choose i}=\frac{(r-i)r(r-1)\cdots(r-i+1)}{i!}$.
\end{center}
We label these terms as $b_{1}, b_{2}, \dots, b_{(r-i) {r \choose i}}$. \\
By comparing the results of these two methods of counting, we
observe that the terms $x_{j_1}x_{j_2} \cdots x_{j_i}x_m $ of $f
\times \ell$, have coefficient equal to $i+1$ in $A_{i+1}$, because
$b_{1}+ b_{2}+ \dots+ b_{(r-i) {r \choose i}}=(i+1)( a_{1}+ a_{2}+
\dots+ a_{{r \choose i+1}})=(i+1)M$. Therefore, while $f \neq 0$, $f
\times \ell =0$ modulo $I$,  whenever $\mathit{p}$ is a prime
divisor of $ i+1$. Hence, $A$ does not have the WLP in these cases.
\end{proof}
The following corollary, not only shows that the WLP may fail for
only finitely many prime numbers, but also it determines these
primes exactly.
\begin{cor}\label{Cor38}
 Let $I^{'}$ be an ideal as  in
 Proposition \ref{1}.
 Then $R/I^{'}$ has the WLP
 whenever $\mathrm{char}\ \mathbb{K}=\mathit{p}$ is not a prime number less than $ i+2$, where
  $1 \leq i \leq \lceil r/2 \rceil$.
\end{cor}
\begin{proof}
Two cases may arise. In the first case, which $i$ varies in the
range $1 \leq i \leq d-1 < \lceil r/2 \rceil $, we only need to
prove the injectivity of the map $\times \ell:A_i \longrightarrow
A_{i+1}$ for every $i$. But this follows from Proposition
\ref{1}, since all  maps $\times \ell:A_i \longrightarrow A_{i+1}$, 
are injective.

In the other case,  $i$ varies in the range  $\lceil r/2 \rceil \leq
i \leq d-1$, and it is enough to prove that all maps $\times \ell:
A_{ i} \longrightarrow A_{i+1}$,
 are surjective. By \cite [Proposition 2.1]{JM} it is enough to do it for $i=\lceil r/2 \rceil $. Let
$x_{j_1}x_{j_2} \cdots x_{j_{i+1}} $ be an element of $A_{i+1}$.
Then it is clear that, it is the image of $f=\sum_{m=1}^i
x_{j_1}\cdots x_{j_{m-1}}\widehat{x_{j_m}}x_{j_{m+1}}\cdots
x_{j_{i+1}}$ under the multiplication map by $\ell$.
\end{proof}

With a method similar to the proofs of Theorems \ref{M} and
\ref{M2}, we can prove the failure of the  WLP for another special
class of monomial ideals.

\begin{thm}({\bf Theorem C.})\label{C}
Let $J$ be the following ideal  of $R$ for which $r\ge 4$ and
$\alpha \geq 5$.
\begin{align*}
J=\left( x_1^{\alpha}, \dots, x_r^{\alpha},
x_1^{\alpha-2}x_2^{2},x_1^2x_2^{\alpha -2},
\dots,
x_{r-1}^{\alpha-2}{x_{r}}^{2},x_{r-1}^{2}{x_{r}}^{\alpha-2}
\right).
\end{align*}
 If $\mathrm{char\ }
 \mathbb{K}=2$, then
 $R/J$ does not have the WLP.
\end{thm}
\begin{proof}
According to Remark \ref{special}, $h_{\alpha-1} \le h_{\alpha}$. We
show that the map $\times \ell: A_{\alpha-1} \longrightarrow
A_\alpha$ is not injective. Let
 $f\in A_{\alpha-1}$ be as follow
\begin{align*}
f=\sum_{1\le i<m\le r}x_i^{\alpha-j}x_m^{j-1} \ \ \mathrm{ where \ }
j=1,2, \alpha-1.
\end{align*}
Then terms of  $f \times \ell$, with respect to different values of
$j$, can be grouped together as follows.
$$
\begin{array}{llllll}
 \mathrm{ for  \ } j=1, & x_i^\alpha, & x_i^{\alpha-1}x_m,&
x_kx_i^{\alpha-1},&\mathrm { where \ }  1 \le  k<  i < m \le r;  \\
\mathrm{ for \  } j=2,& x_i^{\alpha-1}x_m,&   x_i^{\alpha-2}x_m^2,&  x_i^{\alpha-2}x_mx_t,&  \mathrm{where \ }  t\neq i, m, 1 \le  t \le r;   \\
 \mathrm{ for \ } j=\alpha-1,& x_ix_m^{\alpha-1},&
x_i^2x_m^{\alpha-2},& x_ix_m^{\alpha-2}x_t,&   \mathrm{where \ }
t\neq i, m, 1 \le  t \le r.
\end{array}
$$
 The monomials $x_i^\alpha$, $x_i^2x_m^{\alpha-2}$, and $x_i^{\alpha-2}x_m^2$ are
 zero modulo $ I$. As the above list shows,  the monomials $x_i^{\alpha-1}x_m$ appear
 for $j=1$ and $j=2$. Hence, the sum of these monomials in $f \times \ell$  has a coefficient equal to 2, which by
 our assumption, this sum would become zero. As well as, whenever $j=\alpha-1$, each monomial $x_ix_m^{\alpha-1}$, where
 $1 \le i < j\le r$,  is equal to a monomial $x_kx_i^{\alpha-1}$ with $1\le k<i<m\le r$ and $j=\alpha-1$, and vice versa.
  Hence, their sum in $f \times \ell$ is
a multiple of 2. Finally,  for $j=2$ and $j=\alpha-1$ the monomials
$x_i^{\alpha-2}x_mx_t$  and $x_ix_m^{\alpha-2}x_t$ remain in $f
\times \ell$. We  count the number of occurrence of these monomials
in $f \times \ell$ in two different ways.

 {\it First method.} The number of distinct monomials in three distinct
variables of $\{x_1, \dots, x_r\}$ is equal to ${r \choose 3}$, and
since the exponent of only one of its variables is equal to
$\alpha-2$, hence the total number of these elements are equal to
\begin{align*}
3 {r \choose 3}=\frac{r(r-1)(r-2)}{2}.
\end{align*}
We label these monomials as $a_{1}, a_{2}, \dots, a_{3 {r \choose
3}}$ and set $M=a_{1}+ a_{2}+ \dots+ a_{3 {r \choose 3}}$.

{\it Second method.} Each of one of the monomials
$x_i^{\alpha-2}x_mx_t$ and $x_ix_m^{\alpha-2}x_t$ are obtained by
multiplying $x_t\in \{x_1, \dots, x_r\}$, where $ t\neq i, m$, to
terms  $x_i^{\alpha-2}x_m$ and $x_ix_m^{\alpha-2}$  of $f$. Hence,
the number of these monomials  is equal to $2(r-2){r \choose
2}=r(r-1)(r-2)$. We label these terms as $b_{1}, b_{2}, \dots,
b_{2(r-2){r \choose 2}}$.

By comparing the results of these two methods of counting, one can
observe that all monomials in  $f \times \ell$, have a coefficient
equal to two, since $b_{1}+ b_{2}+ \dots+ b_{2(r-2){r \choose
2}}=2(a_{1}+ a_{2}+ \dots+ a_{3 {r \choose 3}})=2M$. Therefore, $f
\times \ell=0$, and the weak Lefschetz property does not hold for
this algebra.
\end{proof}

\begin{rem}\label{special}
Let $J$ be as the above ideal and $A=\bigoplus_{i=0}^e A_i$ with
$h_i=\dim A_i$. In the above argument, we need to have $h_{\alpha-1}
\le h_{\alpha}$. This inequality does not always hold. In fact, from
the structures of $A_{\alpha-1}$ and $A_{\alpha}$, we can deduce
\begin{align*}
h_{\alpha-1}= {\alpha+r-2 \choose r-1} \le h_{\alpha}={\alpha + r-1
\choose r-1} -r -2{r \choose 2},
\end{align*}
where for calculating $h_{\alpha}$, the total number of
$x_i^{\alpha}s$, which is equal to $r$, and the total number of
monomials in the forms $x_i^{\alpha-2}x_j^{2}$ and $x_i^2
x_j^{\alpha -2}$, which  is equal to $2 {r \choose 2}$, should be
subtracted. The above inequality implies $r^2 \le
\frac{(\alpha+r-2)!}{(r-2)!\alpha!} $. But this inequality holds if
$r\ge 4$ and $\alpha\ge 5$.
\end{rem}

According to the method of proofs of the above theorems, we are able
to determine many examples of  monomial Artinian algebras without
the WLP. On the other hand, it is possible to specify some classes
of them with the WLP.
\begin{exm}
Consider the ideal $$I=(x_1^4, x_2^4, x_3^3, x_4^3, x_5^2,
x_1x_2x_3x_4x_5),$$ in $\mathbb{K}[x_1, \dots, x_5]$. The {\bf
h}-vector of $R/I$ is $(1,5,14,28,43,52,49,35,18,6,1)$. Then by
Theorem \ref{M}, $ A_{3} \longrightarrow A_{4 }$ is not injective.
Therefore, it doesn't have the WLP whenever the characteristic of
$\mathbb{K}$ is two.

\end{exm}
\begin{exm}
Let
$I=(x_1^5,x_2^5,x_3^5,x_4^5,x_1^3x_2^2,x_1^2x_2^2,\dots,x_3^3x_4^2,x_3^3x_4^2)$
be an ideal in $\mathbb{K}[x_1,x_2,x_3,x_4]$. Its $\textbf{h}$
vector is $(1, 4, 10, 20, 35, 40, 26, 8, 1)$. Then by Theorem
\ref{C}, $ A_{4} \longrightarrow A_5$  is not injective, so it
doesn't have the WLP whenever the characteristic of $\mathbb{K}$ is
two.
\end{exm}

 \today
\end{document}